\documentclass[11pt]{amsart}
\usepackage{amssymb}
\usepackage{mathrsfs}
\usepackage[arrow,matrix,curve,cmtip,ps]{xy}

\usepackage{tikz}
\usepackage{xcolor}

\allowdisplaybreaks

\usepackage{fixme}

\newtheorem{theorem}{Theorem}[section]
\newtheorem{lemma}[theorem]{Lemma}
\newtheorem{proposition}[theorem]{Proposition}
\newtheorem{corollary}[theorem]{Corollary}
\newtheorem*{theorem*}{Theorem}
\theoremstyle{remark}
\newtheorem{remark}[theorem]{Remark}
\newtheorem{definition}[theorem]{Definition}
\newtheorem{example}[theorem]{Example}

\newtheorem{assumption}[theorem]{Assumption}

\numberwithin{equation}{section}

\newcommand{\N}{\mathbb{N}}

\newcommand{\im}{\operatorname{im }}

\newcommand{\gae}{\lower 2pt \hbox{$\, \buildrel {\scriptstyle >}\over {\scriptstyle
\sim}\,$}}

\newcommand{\lae}{\lower 2pt \hbox{$\, \buildrel {\scriptstyle <}\over {\scriptstyle
\sim}\,$}}

\newcommand{\MU}[1]{
\setbox0\hbox{$#1$}
\setbox1\hbox{$W$}
\ifdim\wd0>\wd1 #1^{\sim} \else \widetilde{#1} \fi
}

\title{Corners of Cuntz-Krieger algebras}

\author{Sara E.~Arklint}
\address{Department of Mathematical Sciences, University of Copenhagen, Uni\-versi\-tets\-parken~5, DK-2100 Copenhagen, Denmark }
\email{arklint@math.ku.dk}
 
\author{Efren Ruiz}
\address{Department of Mathematics, University of Hawaii, Hilo, 200 W.~Kawili St., Hilo, Hawaii, 96720-4091 USA}
\email{ruize@hawaii.edu}

\date{\today}


\begin{document}

\begin{abstract}
We show that 
if $A$ is a unital $C^{*}$-algebra and $B$ is a Cuntz-Krieger algebra 
for which $A \otimes \mathbb{K} \cong B \otimes \mathbb{K}$, then $A$ is a Cuntz-Krieger algebra. 
Consequently, corners of Cuntz-Krieger algebras are Cuntz-Krieger algebras.
\end{abstract}

\maketitle

\section{Introduction}

The Cuntz-Krieger algebras were introduced by J.~Cuntz and W.~Krieger in 1980, \cite{cuntzkrieger}, as $C^{*}$-algebras arising from dynamical systems.
This class of $C^{*}$-algebras has since shown up in several contexts, including the classification program as the Cuntz-Krieger algebras with finitely many ideals, are examples of non-simple purely infinite $C^{*}$-algebras.
It has been known since M.~Enomoto and Y.~Watatani introduced graph algebras in 1980 in \cite{ew},
 that Cuntz-Krieger algebras are the graph algebras arising from finite graphs with no sinks and no sources (see also \cite{mrs}), but no characterization in terms of outer properties has been established for the Cuntz-Krieger algebras.

We show in Theorem~\ref{t:phantom-ck-algebras} that 
the Cuntz-Krieger algebras are the graph algebras arising from finite graphs with no sinks, and conclude that
a graph algebra is a Cuntz-Krieger algebra if and only if it is unital and the rank of its $K_{0}$-group equals the rank of its $K_{1}$-group.
Using this, we show in~Theorem~\ref{t:morita-ck-algebras}, that if a unital $C^{*}$-algebra is stably isomorphic to a Cuntz-Krieger algebra, then it is isomorphic to a Cuntz-Krieger algebra.

As a corollary to Theorem~\ref{t:morita-ck-algebras}, we see that corners of Cuntz-Krieger algebras are Cuntz-Krieger algebras, see Corollary~\ref{cor}.  It is quite surprising that the class of Cuntz-Krieger algebras has this permanence property since the larger class of graph algebras do not (as the graph algebra $\mathsf M_{2^{\infty}}\otimes\mathbb K$ provides a counterexample).  
Moreover, this shows that corners of Cuntz-Krieger algebras are semiprojective, Corollary~\ref{c:sem}, as Cuntz-Krieger algebras are semiprojective.  Our results also show that a unital corner of a stabilized Cuntz-Krieger algebra is semiprojective since a stabilized Cuntz-Krieger algebra is semiprojective.  It was conjectured by B.~Blackadar in \cite[Conjecture~4.4]{blackadar-semiprojective} that a full corner of a semiprojective $C^{*}$-algebra is semiprojective.  He showed in \cite[Proposition~2.7]{blackadar-semiprojective} that a full unital corner of a semiprojective $C^{*}$-algebra is semiprojective.  Recently, S.~Eilers and T.~Katsura showed in~\cite{ek_graph}  that a corner of a unital graph $C^{*}$-algebra that is semiprojective is also semiprojective.
Corollary~\ref{c:sem} is a special case of their results since every Cuntz-Krieger algebra is isomorphic to a unital semiprojective graph $C^{*}$-algebra.  Semiprojectivity is easy in our case since the graphs are finite.  Thus we do not need any results from  \cite{ek_graph}.

\section{Definitions and preliminaries}

\begin{definition} \label{def:graph}
Let $E = (E^0,E^1,s_{E},r_{E})$ be a countable directed graph. A Cuntz-Krieger $E$-family is a set of mutually orthogonal projections $\{ p_v \mid v \in E^0 \}$ and a set $\{ s_e \mid e \in E^1 \}$ of partial isometries satisfying the following conditions:
\begin{itemize}
	\item[(CK0)] $s_e^* s_f = 0$ if $e,f \in E^1$ and $e \neq f$,
	\item[(CK1)] $s_e^* s_e = p_{r_{E}(e)}$ for all $e \in E^1$,
	\item[(CK2)] $s_e s_e^* \leq p_{s_{E}(e)}$ for all $e \in E^1$, and,
	\item[(CK3)] $p_v = \sum_{e \in s_{E}^{-1}(v)} s_e s_e^*$ for all $v \in E^0$ with $0 < |s_{E}^{-1}(v)| < \infty$.
\end{itemize}
The \emph{graph algebra} $C^*(E)$ is defined as the universal $C^*$-algebra given by these generators and relations.
\end{definition}

\begin{definition}
Let $E$ be a directed graph, and let $v\in E^0$ be a vertex in $E$.
The vertex $v$ is called \emph{regular} if $s_E^{-1}(v)$ is finite and nonempty.
If $s_E^{-1}(v)$ is empty, $v$ is called a \emph{sink}, and if $r_E^{-1}(v)$ is empty, $v$ is called a \emph{source}.
If $s_E^{-1}(v)$ is infinite, $v$ is called an \emph{infinite emitter}.
\end{definition}

\begin{definition}\label{d:ckalgebras}
A graph $C^{*}$-algebra of a finite graph with no sinks and no sources is called a \emph{Cuntz-Krieger algebra}.
\end{definition}

J.~Cuntz and W.~Krieger originally defined a Cuntz-Krieger algebra as the universal $C^*$-algebra determined by a collection of partial isometries satisfying relations determined by a finite matrix with entries in $\{0, 1 \}$.  It follows from \cite[Section~4]{KPRR97} that the class of Cuntz-Krieger algebras coincides with the class of graph $C^*$-algebras of finite graphs with no sinks or sources, and moreover, if $E$ is a finite graph with no sinks or sources, $C^*(E)$ coincides with the Cuntz-Krieger algebra associated with the edge matrix of $E$.  

In their study of Cuntz-Krieger algebras, Cuntz and Krieger often imposed Condition (I) on their matrices, which is equivalent to imposing Condition (L) on the graph.  In work after Cuntz and Krieger, particularly in \cite{aHR97}, it was shown that Condition (I) was not necessary, and that loops without exits would produce ideals in the associated $C^{*}$-algebra that are Morita equivalent to $C( \mathbb{T})$.  In this paper, we do not assume our matrices satisfy Condition (I).  Thus obtaining results for $C^{*}$-algebras without real rank zero and results where the $C^{*}$-algebras have uncountably many ideals or are commutative.  We use the language of graph $C^{*}$-algebras in order to provide us with nice combinatorial models of Cuntz-Krieger algebras.  Thus, motivating us to define a Cuntz-Krieger algebra as in Definition~\ref{d:ckalgebras}.

\begin{definition}\label{d:hereditary}
Let $E$ be a directed graph.  A path $\alpha=e_1e_2\cdots e_n$ in $E$ with $r_E(\alpha) := r( e_{n} ) = s_{E} ( e_{1} ) =: s_E(\alpha)$ is called a \emph{cycle}.
A cycle $\alpha=e_1e_2\cdots e_n$ is called \emph{vertex-simple} if $s_{E} ( e_{i} ) \neq s_{E} ( e_{j} )$ for all $i \neq j$.

We refer to $s_E(\alpha)$ as the \emph{base point} of the cycle $\alpha$.
In particular, an edge $e$ in $E$ with $s_E(e)=r_E(e)$ is called a \emph{cycle of length one with base point $s_E(e)$}.
\end{definition}
%

\begin{definition}
Let $E$ be a directed graph.  For vertices $v,w$ in $E$, we write $v\geq w$ if there is a path in $E$ from $v$ to $w$, i.e., a path $\alpha$ in $E$ with $s_E(\alpha)=v$ and $r_E(\alpha)=w$.  Let $S$ be a subset of $E^{0}$.  We write $v \geq S$ if there exists $u \in S$ such that $v \geq u$. 

Let $H$ be a subset of $E^{0}$.  The subset $H$ is called \emph{hereditary} if for all $v\in H$ and $w\in E^0$, $v\geq w$ implies $w\in H$ and $H$ is called \emph{saturated} if $r_E(s_E^{-1}(v))\subseteq H$ implies $v\in H$ for all regular vertices $v$ in $E$.  

For a hereditary subset $H$ in $E^0$, we let $I_H$ denote the ideal in $C^*(E)$ generated by $\{ p_{v} \mid v \in H \}$.
\end{definition}
%

\begin{definition}
Let $E$ be a countable directed graph.  Let $\gamma$ denote the gauge action on $C^*(E)$, i.e., the action $\gamma$ of the circle group $\mathbb T$ on $C^*(E)$ for which $\gamma_z(s_e)=zs_e$ and $\gamma_z(p_v)=p_v$ for all $z\in\mathbb T$, $e\in E^1$, and $v\in E^0$.
An ideal $I$ in $C^*(E)$ is called \emph{gauge invariant} if $\gamma_z(I)\subseteq I$ for all $z\in\mathbb T$.
\end{definition}
%

When $E$ is a row-finite graph, the map $H\mapsto I_H$ defines a lattice isomorphism between the saturated hereditary subsets in $E^0$ and the gauge invariant ideals in $C^*(E)$, see~\cite[Theorem~4.1]{BPRS00}.
%

\begin{definition}
Let $E$ and $F$ be directed graphs.  A \emph{graph homomorphism} $f\colon E\to F$ consists of two maps $f^0\colon E^0\to F^0$ and $f^1\colon E^1\to F^1$ satisfying $r_F \circ f^1=f^0 \circ r_E$ and $s_F \circ f^1=f^0 \circ s_E$.  A graph homomorphim $f\colon E\to F$ is called a \emph{CK-morphism} if $f^0$ and $f^1$ are injective and $f^1$ restricts to a bijection from $s_E^{-1}(v)$ onto $s_F^{-1}(f^0(v))$ for all regular vertices $v$ in $E$.

If $E$ is a subgraph of $F$, we call it a \emph{CK-subgraph} if the inclusion $E\to F$ is a CK-morphism.
\end{definition}

The definition of a CK-morphism between arbitrary graphs was introduced by K.~R.~Goodearl in \cite{goodearl}.  Let \textbf{CKGr} be the category whose objects are arbitrary directed graphs and whose morphisms are CK-morphisms.  Goodearl showed that there is a functor $L_{K}$ from the category \textbf{CKGr}  to the category of algebras over a field $K$.  The functor $L_{K}$ assigns an object $E$ the Leavitt path algebra $L_{K} (E)$.  Goodearl also proved in \cite[Corollary~3.3]{goodearl} that for every CK-morphism $\phi$, the $K$-algebra homomorphism $L_{K} ( \phi )$ is injective.  We now prove the analog of \cite[Corollary~3.3]{goodearl} where the category of $K$-algebras is replaced by the category of $C^{*}$-algebras and the functor assigns an object $E$ the graph $C^{*}$-algebra $C^{*} (E)$.


\begin{lemma} \label{l:ck-subgraph}
Let $E$ and $F$ be countable directed graphs, let $f\colon E\to F$ be a CK-morphism.  Let $\{p_v, s_e \mid v\in E^0, e\in E^1 \}$ be a universal Cuntz-Krieger $E$-family generating $C^*(E)$, and let $\{q_v, t_e \mid v\in F^0, e\in F^1 \}$ be a universal Cuntz-Krieger $F$-family generating $C^*(F)$.

Then the assignments, $p_v\mapsto q_{f^0(v)}$ and $s_e\mapsto t_{f^1(e)}$, induce an injective $*$-homomorphism $\phi\colon C^*(E)\to C^*(F)$ with image equal to the subalgebra of $C^{*} (F)$ generated by $\{ q_v, t_e \mid v,s_F(e)\in f^0(E^0) \}$.
\end{lemma}
\begin{proof}
Using the fact that $f$ is a CK-morphism, one can verify that $\{q_{f^0(v)}, t_{f^1(e)} \mid v\in E^0, e\in E^1 \}$ is a Cuntz-Krieger $E$-family in $C^*(F)$.  The universal property of $C^{*} (E)$ now implies that the $*$-homomorphism $\phi$ exists.  Since $\phi$ intertwines the canonical gauge actions on $C^{*} (E)$ and $C^{*} (F)$ and since $\phi ( p_{v} ) = q_{ f^{0}(v) } \neq 0$ for all $v \in E^{0}$, the gauge invariant uniqueness theorem implies that $\phi$ is injective.  Since $f$ is a CK-morphism, the sets $f^1(E^1)$ and $\{e\in F^1 \mid s_F(e)\in f^0(E^0)\}$ coincide.  It now follows that $\phi ( C^{*} (E) )$ is equal to the subalgebra of $C^{*} (F)$ generated by $\{ q_v, t_e \mid v,s_F(e)\in f^0(E^0) \}$. 
\end{proof}

\section{Graph $C^{*}$-algebras over finite graphs with no sinks}

\begin{assumption}
Throughout the rest of the paper, unless stated otherwise, all graphs will be countable and directed.
\end{assumption}

\begin{definition}
Let $E$ be a graph, let $v_0 \in E^0$ be a vertex, and let $n$ be a positive integer.  Define a graph $E(v_0,n)$ as follows:
\begin{align*}
E(v_0,n)^{0} &= E^{0} \cup \{ v_{1} ,v_{2} , \dots, v_{n} \} \\
E(v_0,n)^1 &= E^{1} \cup \{ e_{1} , e_{2}, \dots, e_{n} \}
\end{align*}
where $r_{E(v_0,n)}$ and $s_{E(v_0,n)}$ extends $r_{E}$ and $s_{E}$ respectively and $r_{E(v_0,n)} ( e_{i} ) = v_{i-1}$ and $s_{E(v_0,n)} ( e_{i} ) = v_{i}$.  
\end{definition}

\begin{definition}
Let $E$ be a graph, let $e_0\in E^1$ be an edge, and let $n$ be a positive integer.  Define a graph $E(e_0,n)$ as follows:
\begin{align*}
E(e_0,n)^{0} &= E^{0} \cup \{ v_{1} , v_{2} , \dots, v_{n} \} \\
E(e_0,n)^{1} &= \left( E^{1} \setminus \{ e_{0} \} \right) \cup \{ e_{1} , e_{2} , \dots, e_{n+1} \} 
\end{align*}
where $r_{E(e_0,n)}$ and $s_{E(e_0,n)}$ extends $r_{E}$ and $s_{E}$ respectively, $r_{E(e_0,n)} ( e_{i} ) = v_{i-1}$ for $i=2,\dots, n+1$ and $s_{E(e_0,n)} ( e_{i} ) = v_{i}$ for $i = 1, \dots, n$, and $r_{E(e_0,n)} ( e_{1} ) = r_E(e_0)$ and $s_{E(e_0,n)} ( e_{n+1} ) = s_{E} ( e_{0} )$.  
\end{definition}

\begin{example} \label{example}
Let $E$ be the graph
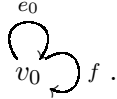
\begin{align*}
\xymatrix{
v_{0} \ar@(ul, ur)[]^{e_{0}} \ar@(ur,dr)[]^{f}
}.
\end{align*} 
Then $E(v_0,n)$ is the graph 
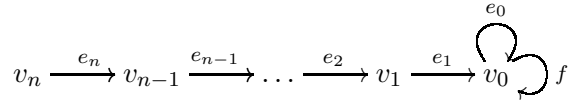
\begin{align*}
\xymatrix{
v_{n} \ar[r]^{e_{n}} \ar[r] & v_{n-1} \ar[r]^{ e_{n-1}} \ar[r] & \dots \ar[r]^{e_{2}} & v_{1}\ar[r]^{ e_{1} } & v_{0}  \ar@(ul, ur)[]^{e_{0}} \ar@(ur,dr)[]^{f}
}
\end{align*}
and $E(e_0,n)$ is the graph
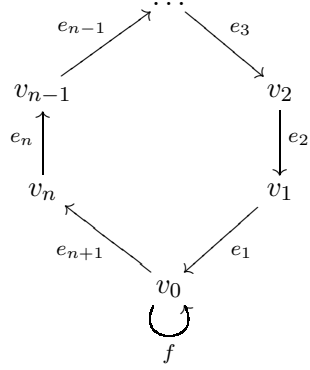
\begin{align*}
\xymatrix{
					& \dots  \ar[rd]^{ e_{3} } 		& \\
v_{n-1} \ar[ru]^{ e_{n-1} } &           			& v_{2} \ar[d]^{ e_{2} } \\
v_{n}	 \ar[u]^{e_{n}}		&					&  v_{1}  \ar[ld]^{ e_{1} } 	\\
					& v_{0}  \ar[ul]^{ e_{n+1} } \ar@(dl,dr)[]_{f}
}
\end{align*}
\end{example}

\begin{proposition}\label{p:removing-sources}
Let $E$ be a graph, let $e_0\in E^1$ be an edge, and let $n$ be a positive integer.  Define $v_{0} = r_{E} ( e_{0} )$.  Then $C^{*} (E(v_0,n)) \cong C^{*} (E(e_0,n))$.
\end{proposition}

\begin{proof}
Let $\{ s_{e} , p_{v} \mid e \in E(e_0,n)^{1} , v \in E(e_0,n)^{0} \}$ be a universal Cuntz-Krieger $E(e_0,n)$-family generating $C^{*} ( E(e_0,n) )$.  For each $v \in E(v_0,n)^{0}$ and $e \in E(v_0,n)^{1}$ set 
\begin{align*}
Q_{v} &= p_{v} \\
T_{e} &= 
\begin{cases}
s_{e}	, &\text{if $e \neq e_{0}$} \\
s_{ e_{n+1} } s_{e_{n}} \cdots s_{ e_{1} }, &\text{if $e = e_{0}$}. 
\end{cases}
\end{align*}

We will show that $\{ T_{e} , Q_{v} \mid e \in E(v_0,n)^{1} , v \in E(v_0,n)^{0} \}$ is a Cuntz-Krieger $E(v_0,n)$-family that generates $C^{*} ( E(e_0,n) )$.  It is clear that $Q_{v} Q_{w} = 0$ for all $v \neq w$.  Let $e , f \in E( v_{0}, n)^{1}$ with $e \neq f$.  Then 
\begin{align*}
T_{e}^{*} T_{f} &=
\begin{cases}
s_{e}^{*} s_{f} , &\text{if $e \neq e_{0}$ and $f \neq e_{0}$} \\
s_{e_{1}}^{*} s_{e_{2}}^{*} \dots s_{ e_{n+1} }^{*} s_{f} , &\text{if $e = e_{0}$} \\
s_{e}^{*} s_{ e_{n+1} } s_{e_{n}} \cdots s_{ e_{1} }, &\text{if $f = e_{0}$} 
\end{cases} \\
&= 0.
\end{align*}
The last two cases hold true because $g \neq e_{n+1}$ for all $g \in E(v_0,n)^{1}$.

Now let $e \in E(v_0,n)^{1}$.  Then 
\begin{align*}
T_{e}^{*} T_{e} &=
\begin{cases}
s_{e}^{*} s_{e}, &\text{if $e \neq e_{0}$} \\
s_{e_{1}}^{*} s_{e_{2}}^{*} \dots s_{ e_{n+1} }^{*} s_{ e_{n+1} } s_{e_{n}} \cdots s_{ e_{1} }, &\text{if $e = e_{0}$} 
\end{cases} \\
&= 
\begin{cases}
p_{ r_{E(e_0,n)  } ( e )   }, &\text{if $e \neq e_{0}$}  \\
p_{ r_{E(e_0,n)  } ( e_{1} )  }, &\text{if $e = e_{0}$} 
\end{cases} \\
&= p_{ r_{ E(v_0,n)  } ( e )  }   \\
&= Q_{ r_{E(v_0,n)} ( e ) }
\end{align*} 
and
\begin{align*}
T_{e} T_{e}^{*} &=
\begin{cases}
 s_{e}s_{e}^{*}, &\text{if $e \neq e_{0}$} \\
s_{ e_{n+1} } s_{e_{n}} \cdots s_{ e_{1} } s_{e_{1}}^{*} s_{e_{2}}^{*} \dots s_{ e_{n+1} }^{*} , &\text{if $e = e_{0}$} 
\end{cases} \\
&\leq
\begin{cases}
p_{ s_{E(e_0,n) } ( e )  }, &\text{if $e \neq e_{0}$}  \\
p_{ s_{E(e_0,n) } ( e_{n+1} ) }, &\text{if $e = e_{0}$} 
\end{cases} \\
&= 
\begin{cases}
p_{ s_{E(v_0,n) } ( e )  }, &\text{if $e \neq e_{0}$}  \\
p_{ s_{E} ( e_{0} )  }, &\text{if $e = e_{0}$} 
\end{cases} \\
&= Q_{ s_{E(v_0,n)} ( e ) }.
\end{align*} 

Let $v \in E(v_0,n)^{0}$ be a regular vertex.  Note that $v$ is a regular vertex in $E(e_0,n)$.  Suppose $v = v_{i}$ for some $i = 1, \dots, n$.  Then $s_{E(e_0,n)}^{-1} ( v_{i} ) = \{ e_{i} \} = s_{E(v_0,n)}^{-1} ( v_{i} ) $.  Hence,
\begin{align*}
Q_{v}  = Q_{v_{i} } = p_{v_{i}} = s_{e_{i}} s_{e_{i}}^{*} = T_{e_{i}} T_{e_{i}}^{*}.
\end{align*}
Suppose $v \neq v_{i}$ for $i = 1, \dots, n$.  We break this into two cases.  Suppose $e_{n+1} \notin s_{E(e_0,n)}^{-1} ( v )$.  Then $v \neq s_{ E } ( e_{0} )$.  Since $v \neq v_{i}$ for $i = 1, \dots, n$ and $v \neq s_{E} ( e_{0} )$, we have that $s_{ E(v_0,n) }^{-1} ( v  ) \cap \{ e_{0} , e_{1} , \dots, e_{n} \}  = \emptyset$ and  $s_{ E(e_0,n) }^{-1} ( v  ) \cap \{ e_{1} , \dots, e_{n}, e_{n+1} \} = \emptyset$.  Thus, 
\begin{align*}
s_{ E(e_0,n) }^{-1} ( v ) = s_{ E }^{-1} ( v ) = s_{ E(v_0,n)}^{-1} ( v ) .
\end{align*}
Hence,
\begin{align*}
Q_{v} = p_{v} = \sum_{ e \in s_{ E(e_0,n) }^{-1} ( v ) } s_{e} s_{e}^{*} = \sum_{ e \in s_{ E(v_0,n) }^{-1} ( v ) } s_{e} s_{e}^{*} = \sum_{ e \in s_{ E(v_0,n) }^{-1} ( v ) } T_{e} T_{e}^{*}.
\end{align*}

Suppose $e_{n+1} \in s_{E(e_0,n)}^{-1} ( v )$.  Then $v= s_{E(e_0,n)} ( e_{n+1} )= s_{E} ( e_{0} )$, which implies that $e_{0} \in s_{E(v_0,n)}^{-1} ( v )$.  Note that $s_{e_{i}} s_{e_{i}}^{*} = p_{ v_{i} }$ for all $i = 1, 2, \dots, n$.  Thus, \begin{align*}
Q_{v} &= p_{v} = \sum_{ e \in s_{E(e_0,n)}^{-1} ( v ) } s_{e} s_{e}^{*} \\
	&= \sum_{ e \in s_{E}^{-1} ( v )\setminus \{e_{0}\} } s_{e} s_{e}^{*} + s_{ e_{n+1} } s_{ e_{n+1} }^{*} \\
	&=  \sum_{ e \in s_{E}^{-1} ( v )\setminus \{e_{0}\} } s_{e} s_{e}^{*} + s_{ e_{n+1} } p_{v_{n}} s_{ e_{n+1} }^{*} \\
	&=  \sum_{ e \in s_{E}^{-1} ( v )\setminus \{e_{0}\} } s_{e} s_{e}^{*} + s_{ e_{n+1} }  s_{ e_{n} } s_{e_{n} }^{*} s_{ e_{n+1} }^{*} \\
	&\ \vdots \\
	&=  \sum_{ e \in s_{E}^{-1} ( v )\setminus \{e_{0}\} } s_{e} s_{e}^{*} + s_{ e_{n+1} } s_{ e_{n} } \cdots s_{e_{1} } s_{e_{1} }^{*} \dots s_{ e_{n} } ^{*} s_{ e_{n+1} }^{*} \\
	&=  \sum_{ e \in s_{E}^{-1} ( v )\setminus \{e_{0}\} } T_{e} T_{e}^{*} + T_{e_{0}} T_{e_{0}}^{*} \\
	&=  \sum_{ e \in s_{E(v_0,n) }^{-1} ( v ) } T_{e} T_{e}^{*}.
\end{align*}

We have just shown that $\{ T_{e} , Q_{v} \mid e \in E(v_0,n)^{1} , v \in E(v_0,n)^{0} \}$ is a Cuntz-Krieger $E(v_0,n)$-family.  Suppose $\{ t_{e} , q_{v} \mid e \in E(v_0,n)^{1} , v \in E(v_0,n)^{0} \}$ is a universal Cuntz-Krieger $E(v_0,n)$-family generating $C^{*} ( E(v_0,n) )$.  Then there exists a $*$-homomorphism $\psi\colon C^{*} ( E(v_0,n) ) \to C^{*} ( E(e_0,n) ) $ such that 
\begin{align*}
\psi ( q_{v} ) &= Q_{v} \\
\psi ( t_{e} ) &= T_{e}
\end{align*}
for all $e \in E(v_0,n)^{1}$ and $v \in E(v_0,n)^{0}$.

Note that the only generator of $C^{*} ( E(e_0,n) )$ that is not included in 
\begin{align*}
\{ T_{e} , Q_{v} \mid e \in E(v_0,n)^{1} , v \in E(v_0,n)^{0} \}
\end{align*}
is $s_{e_{n+1}}$.  In this case, recall again that 
\begin{align*}
p_{v_{i}} = s_{e_{i}} s_{ e_{i} }^{*}
\end{align*}
for all $i = 1, 2, \dots, n$.  Therefore, 
\begin{align*}
T_{ e_{0} } T_{ e_{1} }^{*} \dots T_{ e_{n}}^{*} &= s_{ e_{n+1} } s_{ e_{n}} \dots s_{e_{1}}s_{ e_{1} }^{*} \dots s_{ e_{n}}^{*}  \\
					&=  s_{ e_{n+1} } s_{ e_{n}} \dots s_{ e_{2} } p_{v_{1}} s_{e_{2}}^{*}\dots s_{ e_{n}}^{*} \\
					&=  s_{ e_{n+1} } s_{ e_{n}} \dots s_{ e_{2} } p_{r_{ E(e_0,n) ( e_{2} ) } }  s_{e_{2}}^{*}\dots s_{ e_{n}}^{*} \\
					&= s_{ e_{n+1} } s_{ e_{n}} \dots s_{ e_{2} } s_{e_{2}}^{*}\dots s_{ e_{n}}^{*} \\
					&\ \vdots \\
					&= s_{ e_{n+1} } s_{ e_{n} } s_{e_{n}}^{*} \\
					&= s_{ e_{n+1} } p_{v_{n}} \\
					&= s_{ e_{n+1} } p_{ r_{E(e_0,n) }( e_{n+1} )  } \\
					&= s_{ e_{n+1}}.
\end{align*}
Hence, $s_{e_{n+1}} \in \psi ( C^{*} ( E(v_0,n) ) )$, which implies that $\psi$ is surjective.  

Note that the cycle structure of $E(v_0,n)$ is determined by the cycle structure of $E$ and vice versa.  Moreover, the cycles of $E(v_0,n)$ with no exits are in one-to-one correspondence to the cycles of $E(e_0,n)$ with no exits.  Let $\alpha = f_{1} f_{2} \cdots f_{m}$ be a vertex-simple cycle in $E(v_0,n)$ with no exits.  Suppose $s_{E(v_0,n)} ( f_{i} ) \neq s_{E(v_0,n)} ( e_{0} )$.  Then $\alpha$ is a vertex-simple cycle in $E(e_0,n)$ with no exits.  Thus, $s_{ \alpha }$ is a unitary in $C^{*}(E(e_0,n))$ with spectrum $\mathbb{T}$.  Hence, 
\begin{align*}
\psi ( t_{\alpha} ) = s_{ \alpha },
\end{align*}
which implies that $\psi ( t_{ \alpha } )$ is a unitary in $C^{*} (E(e_0,n))$ with spectrum $\mathbb{T}$.  Suppose $s_{E(v_0,n)} ( f_{i} ) = s_{E(v_0,n)} ( e_{0} )$.  Then $\alpha = e_{0} f_{2} \cdots f_{n}$ since $\alpha$ is a vertex-simple cycle in $E(v_0,n)$ with no exits.  Note that
\begin{align*}
\psi ( t_{ \alpha } ) = s_{ e_{n+1}} s_{ e_{n}} \cdots s_{ e_{1}} s_{f_{2}} \cdots s_{ f_{n}} = s_{ \beta }
\end{align*} 
and $\beta = e_{n+1} e_{n} \cdots e_{1} f_{2} \cdots f_{n}$ is a vertex-simple cycle in $E(e_0,n)$ with no exits.  Hence,  $\psi ( t_{ \alpha } ) = s_{ \beta }$ is a unitary in $C^{*} ( E(e_0,n) )$ with spectrum $\mathbb{T}$.  

From the above paragraph and the fact that $\psi ( q_{v} ) = p_{v} \neq 0$ for all $v \in E(v_0,n)^{0}$, by Theorem~1.2 of \cite{ws-general-ck-uniquness}, $\psi$ is injective.  Therefore, $\psi$ is an isomorphism.  
\end{proof}

\begin{remark}
Proposition~\ref{p:removing-sources} allows one to remove heads of finite length while preserving isomorphism classes.
\end{remark}

\begin{definition}
Let $E$ be a graph and let $H$ be a hereditary subset of $E^{0}$.  Consider the set
\begin{align*}
F( H ) = \{ \alpha \in E^{*} \mid \alpha = e_{1} e_{2} \dots e_{n} , s_{E} ( e_{n} ) \notin H , r_{E} ( e_{n} ) \in H \}.
\end{align*}
Let $\overline{F} (H)$ be another copy of $F (H)$ and we write $\overline{\alpha}$ for the copy of $\alpha$ in $\overline F (H)$.  Define a graph $E(H)$ as follows:
\begin{align*}
E(H)^{0} &= H \cup F(H) \\
E(H)^{1} &= s_{E}^{-1} (H) \cup \overline{F} (H) 
\end{align*}
and extend $s_{E}$ and $r_{E}$ to $E(H)$ by defining $s_{E(H)} ( \overline{\alpha} ) = \alpha$ and $r_{ E(H) } ( \overline{ \alpha } ) = r_{E} ( \alpha )$.
\end{definition}

Note that $E(H)$ is just the graph $( H , r_{E}^{-1}(H), r_{E}, s_{E} )$ together with a source for each $\alpha \in F(H)$ with exactly one edge from $\alpha$ to $r_{E}( \alpha )$.  

\begin{example}
Suppose $E$ is the graph 
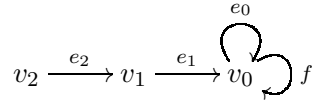
\begin{align*}
\xymatrix{
v_{2} \ar[r]^{e_{2}} & v_{1}\ar[r]^{ e_{1} } & v_{0}  \ar@(ul, ur)[]^{e_{0}} \ar@(ur,dr)[]^{f}
}
\end{align*}
and $H = \{ v_{0} \}$.  Then $F(\{ v_{0} \} ) = \{ e_{1} , e_{2} e_{1} \}$.  Therefore, the graph 
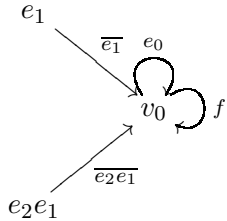
\begin{align*}
\xymatrix{
e_{1} \ar[rd]^-{ \overline{e_{1}} }  & \\
				& v_{0}  \ar@(ul, ur)[]^{e_{0}} \ar@(ur,dr)[]^{f} \\
e_{2} e_{1} \ar[ru]_-{ \overline{ e_{2} e_{1} } } &
}
\end{align*}
represents the graph $E( \{ v_{0} \} )$.
\end{example}

\begin{theorem}\label{t:adding-sources}
Let $E$ be a graph and let $H$ be a hereditary subset of $E^{0}$.  Suppose 
\begin{align*}
( E^{0} \setminus H , r_{E}^{-1} ( E^{0} \setminus H ) , r_{E} , s_{E} )
\end{align*}
is a finite acyclic graph and $v \geq H$ for all $v \in E^{0} \setminus H$.  Assume furthermore that the set $s^{-1}(E^0\setminus H)\cap r^{-1}(H)$ is finite.
Then $C^{*} (E) \cong C^{*} ( E(H) )$.
\end{theorem}

\begin{proof}
Let $\{ s_e, p_v \mid e \in E^1, v \in E^0 \}$ be a universal Cuntz-Krieger $E$-family generating $C^*(E)$.  For $v \in E(H)^0$ define
$$Q_v := \begin{cases} 
p_v & \text{ if $v \in H$} \\
s_\alpha s_\alpha^* & \text{ if $v = \alpha \in F(H)$} \\
\end{cases}$$
and for $e \in E(H)^1$ define
$$T_e := \begin{cases} 
s_e & \text{ if $e \in s_E^{-1}(H)$} \\
s_\alpha & \text{ if $e = \overline{\alpha} \in \overline{F}(H)$} .\\
\end{cases}$$

We shall show that $\{ T_e, Q_v \mid e \in E(H)^1, v \in E(H)^0\}$ is a Cuntz-Krieger $E(H)$-family in $C^{*} (E)$.  To begin, we see that the $Q_v$ are mutually orthogonal projections and the $T_e$ are partial isometries with mutually orthogonal ranges.  (The orthogonality follows from the fact that an element in $F(H)$ cannot extend an element in $F (H )$ and the fact that $s_{ E } ( \alpha ) \notin H$ for all $\alpha \in F(H)$.)

To see the Cuntz-Krieger relations hold, we consider cases for $e \in E(H)^1$.  If $e \in s_{E}^{-1} (H)$, then $r_{H}(e) \in H$ and
$$T_e^* T_e = s_e^*s_e = p_{r_{E}(e)} = Q_{r_{E(H)}(e)}.$$
If $e = \overline{\alpha} \in F(H)$, then $r_{E}(\alpha) \in H$ and
$$T_e^* T_e = T_{\overline{\alpha}}^* T_{\overline{\alpha}} = s_\alpha^*s_\alpha = p_{r_{E}(\alpha)} = Q_{r_{E}(\alpha)} = Q_{r_{E(H)}(\overline{\alpha})} = Q_{r_{ E(H) }(e)}.$$

For the second Cuntz-Krieger relation, we again let $e \in E(H)^1$ and consider cases.  If $e \in s_{E}^{-1} (H)$, then 
$$Q_{s_{E(H)}(e)} T_e = p_{s_{E}(e)} s_e = s_e = T_e.$$
If $e = \overline{\alpha} \in \overline{F}(H)$, then 
$$Q_{s_{E(H)}(e)} T_e = Q_\alpha T_{\overline{\alpha}} = s_\alpha s_\alpha^* s_\alpha = s_\alpha = T_{\overline{\alpha}} = T_e.$$
Thus $Q_{s_{E(H)}(e)} T_e = T_e$ for all $e \in E(H)^1$, so that $T_e T_e^* \leq Q_{s_{E(H)}(e)}$ for all $e \in 
E(H)^1$, and the second Cuntz-Krieger relation holds.

For the third Cuntz-Krieger relation, suppose that  $v \in E(H)^0$ and that $v$ is regular.  If $v \in H$, then the set of edges that $v$ emits in $E(H)$ is equal to the set of edges that $v$ emits in $E$, and hence 
$$Q_v = p_v = \sum_{ \{e \in E^1 \mid s_{E}(e) = v \} } s_es_e^* = \sum_{ \{e \in E(H)^{1} \mid s_{E(H)}(e) = v\} } T_eT_e^*.$$
If $v \in F(H)$, then $v = \alpha$ with $r_{E(H)}(\alpha) \in H$, and the element $\overline{\alpha}$ is the unique edge in $E(H)^0$ with source $v$, so that 
$$Q_v = s_\alpha s_\alpha^* = T_e T_e^*.$$
Thus the third Cuntz-Krieger relation holds, and 
\begin{align*}
\{ T_e, Q_v \mid e \in E(H)^1, v \in E(H)^0\}
\end{align*}
is a Cuntz-Krieger $E(H)$-family in $C^{*} (E)$.

If $\{ q_v, t_e \mid v \in E(H)^0, e \in E(H)^1 \}$ is a universal Cuntz-Krieger $E(H)$-family generating $C^*(E(H))$, then by the universal property of $C^*(E(H))$ there exists a $*$-homomorphism $\phi : C^*(E(H)) \to C^{*}(E)$ with $\phi (q_v) = Q_v$ for all  $v \in E(H)^{0}$ and $\phi(t_e) = T_e$ for all  $e \in E(H)^1$. 

We shall show injectivity of $\phi$, by applying the generalized Cuntz-Krieger uniqueness theorem of \cite{ws-general-ck-uniquness}.  To verify the hypotheses, we first see that if $v \in E(H)^0$, then $\phi( q_v) = Q_v \neq 0$.  Second, if $e_1 \ldots e_n$ is a vertex-simple cycle in $E(H)$  with no exits, then since the cycles in $E(H)$ come from cycles in $E$ all lying in the subgraph given by 
\begin{align*}
( H , s_{E}^{-1} (H) , s_{E} , r_{E} ),
\end{align*}
we must have that $e_i \in E^1$ for all $1 \leq i \leq n$, and $e_1 \ldots e_n$ is a cycle in $E$ with no exits.  Thus $\phi(t_{e_1 \ldots e_n}) = \phi(t_{e_1}) \ldots \phi(t_{e_n}) = s_{e_1} \ldots s_{e_n} = s_{e_1 \ldots e_n}$ is a unitary whose spectrum is the entire circle.  It follows from the generalized Cuntz-Krieger uniqueness theorem, stated in Theorem~1.2 of \cite{ws-general-ck-uniquness}, that $\phi$ is injective.

We now show that $\phi$ is surjective.  Let $e \in E^{1}$ such that $r_{E} ( e ) \in H$.  If $s_{E} (e) \in H$, then 
\begin{align*}
s_{e} = T_{{e}} = \phi ( t_{{e} } ) \in \mathrm{im} ( \phi ).
\end{align*}
Suppose $s_{E} ( e ) \notin H$.  So, $e \in F(H)$.  Hence,
\begin{align*}
s_{e} = T_{\overline e} = \phi ( t_{\overline{e} } ) \in \mathrm{im} ( \phi ).
\end{align*}
We now show that $s_e\in\im(\phi)$ for all $e\in r_E^{-1}(E^0\setminus H)$.
By assumption, $v \geq H$ for all $v\in E^0\setminus H$.  Define for each $k$ the subset $D_k$ of $E^0\setminus H$ as the set of vertices $v$ for which $k$ is the maximal length of a path from $v$ to $H$.  
Put $D_0=H$, and note that for $k\geq 1$, all vertices in $D_k$ are regular.
By induction on $k\geq 1$ we will show for every path $\alpha$ in $E$ with $r_E(\alpha)\in D_k$ that $s_\alpha\in\im(\phi)$.

For $k=1$ and $\alpha$ a path in $E$ with $r_E(\alpha)\in D_1$, we note that $r_E(e)\in H$ for all $e\in s_E^{-1}(r_E(\alpha))$. Hence
\begin{align*}
s_\alpha &= s_\alpha p_{r_E(\alpha)} = \sum_{e\in s_E^{-1}(r_E(\alpha))}s_\alpha s_es_e^* \\
&= \sum_{e\in s_E^{-1}(r_E(\alpha))}T_{\overline{\alpha e}}T_{\overline e}^* \\
&= \sum_{e\in s_E^{-1}(r_E(\alpha))}\phi(t_{\overline{\alpha e}}t_{\overline e}^*) \in\im(\phi)
\end{align*}
since $\alpha e,e\in F(H)$.

For $k>1$ and $\alpha$ a path in $E$ with $r_E(\alpha)\in D_k$, we note that for all $e\in s_E^{-1}(r_E(\alpha))$ there is a $j<k$ for which $r_E(\alpha e)=r_E(e)\in D_j$.
Hence 
\[ s_\alpha = \sum_{e\in s_E^{-1}(r_E(\alpha))}s_{\alpha e}s_e^* \in\im(\phi). \]

We have just shown that $s_{e} \in \mathrm{im} ( \phi )$ for all $e \in E^{1}$.  We now show that $p_{v} \in \mathrm{im} ( \phi )$ for all $v \in E^{0}$.  Note that if $v \in E^{0}$ and $v$ is not a regular vertex, then $v \in H$.  Hence, $p_{v} = Q_{v} = \phi ( q_{v} )$.  Let $v \in E^{0}$ be a regular vertex.  Then for each $e \in s_{E}^{-1} ( v )$, we have that $s_{e} , s_{e}^{*} \in \mathrm{im} ( \phi )$.  Therefore,
\begin{align*}
p_{v} = \sum_{ e \in s_{E}^{-1} ( v ) } s_{e} s_{e}^{*} \in \mathrm{im} ( \phi ).
\end{align*}

Since $\{ p_{v}, s_{e} \mid v \in E^{0} , e \in E^{1} \} \subseteq \mathrm{im} ( \phi )$ and $\{ p_{v}, s_{e} \mid v \in E^{0} , e \in E^{1} \}$ generates $C^{*} (E)$ we have that $\phi$ is surjective.  Therefore, $\phi$ is an isomorphism.
\end{proof}

\begin{definition}
Let $E$ be a graph, let $v_0\in E^0$ be a vertex in $E$, and let $n$ be a positive integer.  Define a graph $E'(v_0,n)$ as follows:
\begin{align*}
E'(v_0,n)^{0} &= E^{0} \cup \{ v_{1} ,v_{2} , \dots, v_{n} \} \\
E'(v_0,n)^{1} &= E^{1} \cup \{ e_{1} , e_{2}, \dots, e_{n} \}
\end{align*}
where $r_{E'(v_0,n)}$ and $s_{E'(v_0,n)}$ extends $r_{E}$ and $s_{E}$ respectively, and $r_{E'(v_0,n)} ( e_{i} ) = v_{0}$ and $s_{E'(v_0,n)} ( e_{i} ) = v_{i}$ for all $i=1,\dots n$.
\end{definition}

\begin{corollary}\label{c:adding-sources}
Let $E$ be a graph, let $v_0\in E^0$ be a vertex, and let $n$ be a positive integer.  Then $C^*(E(v_0,n))\cong C^*(E'(v_0,n))$.
\end{corollary}

\begin{proof}
Note that $E^{0}$ is a hereditary subset of $E(v_0,n)^{0}$.  Moreover $(E(v_0,n)^{0} \setminus E^{0} , r_{E(v_0,n)}^{-1} (E(v_0,n)^{0} \setminus E^{0} ) , r_{E(v_0,n)} , s_{E(v_0,n)} )$ is a finite acyclic graph and for each $v \in E(v_0,n)^{0} \setminus E^{0}$, there exists a path in $E(v_0,n)$ from $v$ to $E^{0}$.
Finally,
$s_{E(v_0,n)}(E(v_0,n)^0\setminus E^0)\cap r_{E(v_0,n)}^{-1}(E^0)$
is finite.
Thus, by Theorem~\ref{t:adding-sources}, $C^{*}(E(v_0,n)) \cong C^{*} (E(v_0,n)( E^{0} ) )$.  One can verify that the graph $E(v_0,n)( E^{0} )$ is isomorphic to the graph $E'(v_0,n)$.  Thus, $C^{*} (E(v_0,n)( E^{0} ) ) \cong C^{*} (E'(v_0,n))$.
\end{proof}

\begin{example}\label{e:graphsconstruction}
We give an example to illustrate the proof of Corollary~\ref{c:adding-sources}.
Consider the graph $E$ of Example~\ref{example}.
Then $E(v_0,2)$ is the graph
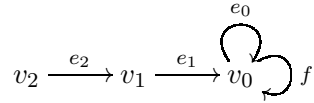
\begin{align*}
\xymatrix{
v_{2} \ar[r]^{e_{2}} & v_{1}\ar[r]^{ e_{1} } & v_{0}  \ar@(ul, ur)[]^{e_{0}} \ar@(ur,dr)[]^{f}
}
\end{align*}
and thereby $E(v_0,2)(\{v_0\})$ is the graph
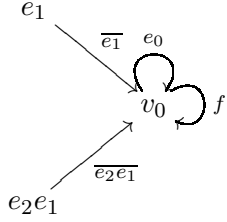
\begin{align*}
\xymatrix{
e_{1} \ar[rd]^-{ \overline{e_{1}} }  & \\
				& v_{0}  \ar@(ul, ur)[]^{e_{0}} \ar@(ur,dr)[]^{f} \\
e_{2} e_{1} \ar[ru]_-{ \overline{ e_{2} e_{1} } } &
}
\end{align*}
which is isomorphic to the graph $E'(v_0,2)$ 
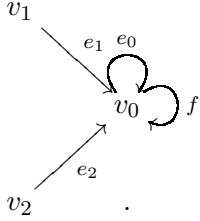
\begin{align*}
\xymatrix{
v_{1} \ar[rd]^-{ e_{1} }  & \\
				& v_{0}  \ar@(ul, ur)[]^{e_{0}} \ar@(ur,dr)[]^{f} \\
v_{2}  \ar[ru]_-{ e_{2} } & .
}
\end{align*}
\end{example}

\begin{theorem}\label{t:phantom-ck-algebras}
Let $E$ be a graph.  Then the following are equivalent:
\begin{itemize}
\item[(1)] $E$ is finite graph with no sinks.

\item[(2)] $C^{*} (E)$ is isomorphic to a Cuntz-Krieger algebra.

\item[(3)] $C^{*} (E)$ is unital and 
\begin{align*}
\mathrm{rank} ( K_{0} ( C^{*} (E) ) ) = \mathrm{rank} ( K_{1} ( C^{*} ( E ) ) ).
\end{align*}
\end{itemize}
\end{theorem}

\begin{proof}
We first show (1) implies (2).  Suppose $E$ is a finite graph with no sinks.  Remove the sources from $E$, and remove the vertices that then become sources; repeat this procedure finitely many times to get a subgraph $F$ of $E$ that has no sinks and no sources.
Notice that $F^{0}$ is a hereditary subset of $E^{0}$, that
\begin{align*}
( E^{0} \setminus F^{0} , r_{E}^{-1} ( E^{0} \setminus F^{0} ) , r_{E} , s_{E} )
\end{align*}
is a finite acyclic graph, and that for each $v \in E^{0} \setminus F^{0}$, there exists a path in $E$ from $v$ to $F^{0}$.
Therefore, by Theorem~\ref{t:adding-sources}, $C^{*} ( E ) \cong C^{*} ( E( F^{0} ) )$.  We can apply Corollary~\ref{c:adding-sources} and Proposition~\ref{p:removing-sources} as many times as needed (but finitely many times), to get a finite graph $E_{1}$ with no sinks and no sources such that $C^{*} ( E( F^{0} ) ) \cong C^{*} (E_{1})$.  Note that $C^{*} (E_{1})$ is a Cuntz-Krieger algebra and $C^{*} (E) \cong C^{*} (E_{1})$.

We next show (2) implies (3).  Suppose $C^{*} (E)$ is isomorphic to a Cuntz-Krieger algebra.  Then $C^{*} (E)$ is unital.  Moreover, by the $K$-theory computation (Theorem~3.1 of \cite{ddmt_kthygraph}), 
\begin{align*}
\mathrm{rank} ( K_{0} ( C^{*} (E) ) ) = \mathrm{rank} ( K_{1} ( C^{*} ( E ) ) ).
\end{align*}

We now show (3) implies (1).  Suppose $C^{*} (E)$ is unital.  Then $E^{0}$ is a finite set.  Since 
\begin{align*}
\mathrm{rank} ( K_{0} ( C^{*} (E) ) ) = \mathrm{rank} ( K_{1} ( C^{*} ( E ) ) ),
\end{align*}
by the $K$-theory computation (Theorem~3.1 of \cite{ddmt_kthygraph}), $E$ has no singular vertices.  Hence, $E$ is a finite graph with no sinks.
\end{proof}

\begin{definition} \label{d:stabilization}
Let $E$ be a graph and let $SE$ be the graph obtained by adding an infinite head to every vertex of $E$.  
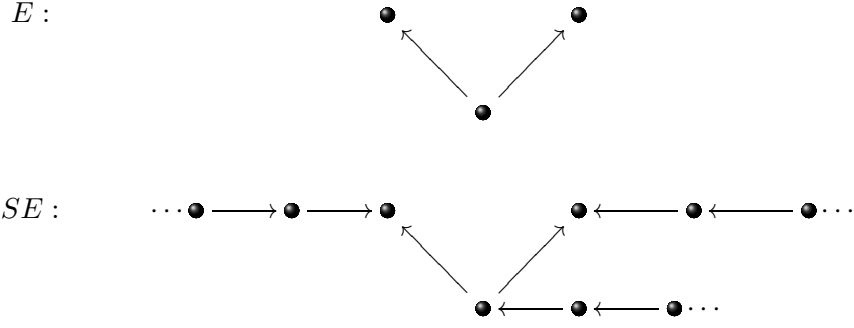
\begin{align*}
 \xymatrix{
E: \ 	&			&	& \tikz \shade[ball color=black] (0,0) circle (1mm);& & \tikz \shade[ball color=black] (0,0) circle (1mm); \\
	&			&	& & \tikz \shade[ball color=black] (0,0) circle (1mm); \ar[ru] \ar[lu] & \\
SE : \ & \cdots \tikz \shade[ball color=black] (0,0) circle (1mm); \ar[r] &  \tikz \shade[ball color=black] (0,0) circle (1mm); \ar[r] & \tikz \shade[ball color=black] (0,0) circle (1mm); & & \tikz \shade[ball color=black] (0,0) circle (1mm); & \ar[l] \tikz \shade[ball color=black] (0,0) circle (1mm);  & \tikz \shade[ball color=black] (0,0) circle (1mm); \ar[l] \cdots \\
	& &  & & \tikz \shade[ball color=black] (0,0) circle (1mm); \ar[ru] \ar[lu] & \ar[l] \tikz \shade[ball color=black] (0,0) circle (1mm); & \tikz \shade[ball color=black] (0,0) circle (1mm); \ar[l] \cdots & \\
}
\end{align*}
We call $SE$ the \emph{stabilization} of $E$.
\end{definition}

\begin{theorem}\label{t:fullcorners-stablized}
Let $E$ be a graph with finitely many vertices and let $T$ be a finite hereditary subset of $(SE)^{0}$ such that $E^{0} \subseteq T$. Set 
\begin{align*}
p_{T} = \sum_{ v \in T } p_{v}
\end{align*}
where $\{ s_{e} , p_{v} \mid e \in (SE)^{1} , v \in ( SE )^{0} \}$ is a universal Cuntz-Krieger $SE$-family generating $C^{*} (SE)$.  Then $p_{T}$ is a full projection in $C^{*} (SE)$ and there exists a subgraph $F$ of $SE$ such that $C^{*} (F) \cong p_{T}C^{*} ( SE )p_{T}$.

If in addition $C^{*} (E)$ is a Cuntz-Krieger algebra, then $p_{T} C^{*} ( SE ) p_{T}$ is a Cuntz-Krieger algebra.
\end{theorem}

\begin{proof}
The smallest saturated subset of $(SE)^{0}$ containing $T$ is $(SE)^{0}$.  Hence, $p_{T}$ is a full projection. 

Let $F = ( T , s_{SE}^{-1} (T) , r_{ SE} , s_{SE} )$.  We claim that $F$ is a 
CK-subgraph of $SE$.  It is clear that $F$ is a subgraph of $SE$.  We will show that $s_{E}^{-1} ( v ) = s_{F}^{-1} ( v )$ for all $v \in F^{0}$.  Let $v \in F^{0}$.  Suppose $v \in E^{0}$.  Then 
\begin{align*}
s_{SE}^{-1} ( v ) = s_{E}^{-1} ( v ) = s_{ F }^{-1} ( v ).
\end{align*}
Suppose $v \in T \setminus E^0$.
Then $s_{SE}^{-1} ( v ) = \{ e \} = s_{F}^{-1} ( v )$ for some $e$.  Since $F$ is a
CK-subgraph of $SE$, we have by Lemma~\ref{l:ck-subgraph}
that $C^{*} (F)$ is isomorphic to the subalgebra of $C^{*} (SE)$ generated by 
\begin{align*}
\{ p_{v} , s_{e} \mid s_{SE} ( e ) , v \in T \},
\end{align*}
which we denote by $B$.  We claim that $p_{T} C^{*} (SE ) p_{T} = B$.  

Note that $B$ is unital with unit $p_{T}$.  Note that if $e \in s_{SE}^{-1} ( T )$, then $s_{SE} ( e )$ and $r_{SE} ( e )$ are elements of  $T$.  Therefore, 
for all $v\in T$ and all $e\in s_{SE}^{-1}(T)$,
\begin{align*}
p_{v} = p_{T} p_{v} p_{T} \in p_{T} C^{*} (SE ) p_{T}  
\end{align*}
and
\begin{align*}
s_{e} = p_{s_{SE} (e) } s_{e} p_{r_{SE} (e) } = p_{T} p_{s_{SE} (e) } s_{e} p_{r_{SE} (e) } p_{T}  \in p_{T} C^{*} (SE ) p_{T}. 
\end{align*}
Hence, $B$ is a subalgebra of $p_{T} C^{*} ( SE ) p_{T}$.  

Let $\alpha$ be a finite path in $SE$.  Suppose $s_{SE} (\alpha )$ is not an element of $T$.  Then 
\begin{align*}
p_{T} p_{ s_{SE} ( \alpha ) } = 0.
\end{align*}
If $s_{SE} (\alpha ) \in T$, then 
\begin{align*}
p_{T} p_{ s_{SE} ( \alpha ) }  = p_{ s_{SE} ( \alpha ) }.
\end{align*}
From these observations, we get that
\begin{align*}
p_{T} s_{\alpha} s_{ \beta }^{*} p_{T} &= 
\begin{cases}
s_{ \alpha } s_{ \beta }^{*}, &\text{if $s_{SE} ( \alpha ), s_{SE} ( \beta ) \in T$} \\
0, &\text{otherwise}.
\end{cases}
\end{align*}
Since $e \in s_{SE}^{-1} ( T )$ implies that $s_{SE} ( e )$ and $r_{SE} ( e )$ are elements of  $T$, we have that $\alpha$ is a path in $F$ if $s_{SE} ( \alpha ) \in T$.  Therefore, if $s_{SE} ( \alpha ), s_{SE} ( \beta ) \in T$, then $s_{\alpha} , s_{\beta}^{*} \in B$.  Hence, 
\begin{align*}
p_{T} s_{\alpha} s_{ \beta }^{*} p_{T} 
\end{align*}
is an element of $B$ for all paths $\alpha$ and $\beta$ in $SE$.  We have just shown that $B = p_{T} C^{*}  ( SE ) p_{T}$ which implies that $C^{*} (F) \cong B = p_{T} C^{*} ( SE ) p_{T}$.

Assume that $C^*(E)$ is isomorphic to a Cuntz-Krieger algebra.  Then by Theorem~\ref{t:phantom-ck-algebras}, the graph $E$ is finite and has no sinks.  Since $F$ is a graph obtained from the graph $E$ by adding a finite head to some vertices of $E$, the graph $F$ is finite and with no sinks.  By Theorem~\ref{t:phantom-ck-algebras}, $C^*(F)$ is a Cuntz-Krieger algebra.
\end{proof}

\section{Unital $C^*$-algebras that are stably isomporphic to Cuntz-Krieger algebras}

\begin{definition}
For a $C^*$-algebra $A$, and projections $p \in \mathsf M_n(A)$ and $q \in \mathsf M_m(A)$, we say $p$ is \emph{Murray-von Neumann equivalent} to $q$, denoted $p \sim q$, if there exists $v \in \mathsf M_{m,n} (A)$ with $p=v^*v$ and $q = vv^*$.

For a projection $p$ in a $C^{*}$-algebra $A$ and $n \in \N$, $n p$ will denote the projection 
\begin{align*}
\underbrace{p \oplus p \cdots \oplus p}_{ \text{$n$-times} } \in \mathsf{M}_{n} (A).
\end{align*}
\end{definition}

\begin{lemma}\label{l:projects}
Let $E$ be a graph and let $\{p_v,s_e \mid v\in E^0, e\in E^1 \}$ denote a universal Cuntz-Krieger $E$-family generating $C^*(E)$.  Let $v\in E^0$ and assume that $v$ is a regular vertex.
Then \[p_v\sim \sum_{e\in s^{-1}(v)}p_{r_E(e)}.\]
\end{lemma}

\begin{proof}
The result follows directly from the Cuntz-Krieger relations, see Definition~\ref{def:graph}.
\end{proof}

\begin{lemma} \label{basiclemma}
Let $E$ be a row-finite graph and let $\{p_v,s_e \mid v\in E^0, e\in E^1 \}$ denote a universal Cuntz-Krieger $E$-family generating $C^*(E)$. 
Let $v,w\in E^0$ with $v \neq w$.  If there is a path from $v$ to $w$ in $E$, then there exists a family $(m_u(v,w))_{u\in E^0}$ of non-negative integers satisfying
\[ p_v\sim p_w + \sum_{u\in E^0} m_u(v,w)p_u \]
with all but finitely many $m_u(v,w)$ equal to zero.
Moreover, $m_{v} (v,w)$ can be chosen such that 
\begin{align*}
m_{v} ( v, w ) \geq | \{ e \in E^{1} \mid s_{E} (e) = r_{E} (e) = v \} |.
\end{align*}
\end{lemma}

\begin{proof}
Let $e_1\cdots e_n$ denote a path in $E$ from $v$ to $w$, so that $e_1,\ldots,e_n\in E^1$ with $r_E(e_i)=s_E(e_{i+1})$ for all $i\in\{1,\ldots, n-1\}$, $s_E(e_1)=v$, and $r_E(e_n)=w$.
Define $v_i=r_E(e_i)$ for $i\in\{1,\ldots,n\}$, and $v_0=v$.
Then by Lemma~\ref{l:projects},
\begin{align*}
 p_v &\sim p_{v_1} + \sum_{e\in s^{-1}(v)\setminus\{e_1\}}p_{r_E(e)} \\
 &\sim p_{v_2} +  \sum_{e\in s^{-1}(v_1)\setminus\{e_2\}}p_{r_E(e)}  + \sum_{e\in s^{-1}(v_0)\setminus\{e_1\}}p_{r_E(e)} \\
 &\ \vdots \\
  &\sim p_w + \sum_{i=1}^{n-1} \left(  \sum_{e\in s^{-1}(v_{i-1}))\setminus\{e_i\}}p_{r_E(e)} \right) .
 \end{align*}
Define $(m_u(v,w))_{u\in E^0}$ as the non-negative integer scalars in the above linear combination of $(p_u)_{u\in E^0}$, i.e., such that
\[ \sum_{i=1}^{n-1} \left(  \sum_{e\in s^{-1}(v_{i-1}))\setminus\{e_i\}}p_{r_E(e)} \right) = \sum_{u\in E^0} m_u(v,w)p_u . \]
This defines $(m_u(v,w))_{u\in E^0}$ for any pair $v,w\in E^0$ for which there is a path from $v$ to $w$.

The last statement is clear from the construction of $m_{u} ( v, w )$.  
\end{proof}

\begin{theorem}[Theorem~3.5 of \cite{amp:nonstablekthy}] \label{amp:nonstablekthy}
Let $E$ be a row-finite graph and let $\{p_v,s_e \mid v\in E^0, e\in E^1 \}$ denote a universal Cuntz-Krieger $E$-family generating $C^*(E)$.
Any projection in $C^{*} (E) \otimes \mathbb{K}$ is Murray-von Neumann equivalent to a projection of the form $\sum_{u\in E^0} m_up_u$ with all but finitely many $m_u$ equal to zero.
\end{theorem}

\begin{lemma}\label{l:supportloops}
Suppose $E$ is a row-finite graph in which every vertex is the base point of at least one cycle of length one.  Then every hereditary subset in $E^0$ is saturated.
\end{lemma}

\begin{proof}
Let $H$ be a hereditary subset in $E^{0}$.  Since every vertex in $E$ is the base point of at least one cycle of length one, $E$ is a graph with no sinks.  This fact and the fact that $E$ is row-finite imply that every vertex in $E$ is a regular vertex.   To show that $H$ is saturated we must show that $r_{E} ( s_{E}^{-1} ( v ) ) \subseteq H$ implies $v \in H$ for all $v \in E^{0}$.  Let $v$ be a vertex in $E$ such that $r_{E} ( s_{E}^{-1} ( v ) ) \subseteq H$.  By assumption there exists $e \in s_{E}^{-1} ( v )$ such that $v = r_{E} (e) = s_{E} (e)$.  Hence, $v \in r_{E} ( s_{E}^{-1} (v) )$ which implies that $v \in H$.
\end{proof}

\begin{lemma}\label{l:fullprojections}
Let $E$ be a finite graph and let $\{p_v,s_e \mid v\in E^0, e\in E^1 \}$ denote a universal Cuntz-Krieger $E$-family generating $C^*(E)$.   Assume that $E$ has no sinks and no sources, and every vertex of $E$ is a base point of at least one cycle of length one.

Let $p$ be a norm-full projection in $C^*(E)\otimes\mathbb K$.  Then there exists a family $(m_u)_{u\in E^0}$ of integers satisfying
\[ p\sim \sum_{u\in E^0} m_up_u  \]
and $m_u\geq 1$ for all $u\in E^0$
\end{lemma}

\begin{proof}
By Theorem~\ref{amp:nonstablekthy}, there exists a family $(n_u)_{u\in E^0}$ of non-negative integers satisfying
\[ p\sim\sum_{u\in E^0}n_up_u .\] 
Set $S_{0} = \{ u \in E^{0} \mid n_{u} \neq 0 \}$ and let $H_{0}$ be the smallest hereditary subset of $E^{0}$ that contains $S_{0}$.  By Lemma~\ref{l:supportloops}, $H_{0}$ is saturated.  Set $q = \sum_{ v \in H_{0} } p_{v} \in I_{H_{0}}$.  Note that the ideal generated by $q \otimes e_{11}$ is equal to the ideal generated by $\sum_{u\in E^0}n_up_u$, where $\{ e_{ij} \}_{i,j}$ is a system of matrix units for $\mathbb{K}$.  Since $ p\sim\sum_{u\in E^0}n_up_u$, we have that the ideal generated by $q \otimes e_{11}$ is equal to the ideal generated by $p$.  Thus, $q \otimes e_{11}$ is a norm-full projection in $C^{*} (E) \otimes \mathbb{K}$ which implies that $q$ is a norm-full projection in $C^{*} (E)$.  Hence, $I_{H_{0}} = C^{*} (E)$ which implies that $H_{0} = E^{0}$.  Therefore, for every $w \in E^{0}$, there exists $v \in S_0$ such that $v \geq w$.

Set $E^{0} \setminus S_{0} = \{ w_{0} , w_{1} , \dots, w_{m} \}$.  Let $v \in S_0$ such that $v \geq w_{0}$.  By Lemma~\ref{basiclemma}, 
\begin{align*}
p_v\sim p_{w_{0}} + \sum_{u\in E^0} m_u(v,w_{0})p_u 
\end{align*}
where $m_{u} ( v, w_{0} ) \geq 0$ and 
\begin{align*}
m_{v} ( v, w_{0} ) \geq | \{ e \in E^{1} \mid s_{E} (e) = r_{E} (e) = v \} | \geq 1.
\end{align*}
Therefore,
\begin{align*}
p \sim \sum_{ u \in E^{0} } n_{u}' p_{u}
\end{align*}
where $n_{u}' \geq 0$ for all $u \in E^{0}$.  Moreover,
\begin{align*}
S_{0} \subsetneq \{ u \in E^{0} \mid n_{u} ' \neq 0 \} = S_{1}
\end{align*}
since $n_{w_{0}}' \neq 0$ but $n_{w_{0}} = 0$.  Therefore, $| E^{0} \setminus S_{1} | < | E^{0} \setminus S_{0} |$.  

Let $H_{1}$ be the smallest hereditary subset of $E^{0}$ that contains $S_{1}$.  Note that $E^{0} = H_{0} \subseteq H_{1} \subseteq E^{0}$.  Hence, $H_{1} = E^{0}$.  Hence, for each $w \in E^{0}$, there exists a $v \in S_{1}$ such that $v \geq w$.  Therefore, we may continue this process to get   
a family $(m_u)_{u\in E^0}$ of non-negative integers satisfying
\[ p\sim \sum_{u\in E^0} m_up_u  \]
and $m_u\geq 1$ for all $u\in E^0$.
\end{proof}

\begin{proposition}\label{p:full-corner-CK-algebras}
Let $E$ be a finite graph with no sinks and no sources, and 
assume that
every vertex of $E$ is a base point of at least one cycle of length one.  Let $p$ be a norm-full projection in $C^{*} (E) \otimes \mathbb{K}$.  Then there exists a finite graph $F$ that has no sinks and no sources such that $C^*(F)\cong p ( C^{*} (E) \otimes \mathbb{K} ) p$.
\end{proposition}

\begin{proof}
Let $SE$ be the stabilization of $E$, as defined in Definition~\ref{d:stabilization}.  Let $\{ e_{ij} \}_{ i ,j}$ be a system of matrix units for $\mathbb{K}$.  By Proposition~9.8 of \cite{gamt:isomorita} and its proof, there exists an isomorphism $\phi \colon  C^{*} (E) \otimes \mathbb{K} \rightarrow C^{*} ( SE ) $ such that 
\begin{align*}
K_{0} ( \phi ) ( [ p_{v}\otimes e_{11} ] ) = [ p_{v} ]
\end{align*}
for all $v \in E^{0}$.  Let $p$ be a norm-full projection in $C^{*} (E) \otimes \mathbb{K}$.  By Lemma~\ref{l:fullprojections}, $p$ is Murray-von Neumann equivalent to $\sum_{u\in E^0} m_up_u$ with $m_u\geq 1$ for all $u\in E^0$.
Therefore, since $C^*(SE)$ has weak cancellation by Corollary~7.2 of~\cite{amp:nonstablekthy},
$\phi(p)$ is Murray-von Neumann equivalent to $p_{T} \in C^{*} (SE)$ such that $T$ is a finite, hereditary subset of $(SE)^{0}$ with $E^{0} \subseteq T$.  By Theorems~\ref{t:fullcorners-stablized} and~\ref{t:phantom-ck-algebras}, $p_{T} C^{*} (SE) p_{T} \cong C^{*} (F)$ for some finite graph $F$ with no sinks and no sources.  Note that $p ( C^{*} (E) \otimes \mathbb{K} ) p \cong \phi ( p ) C^{*} (SE) \phi ( p ) \cong p_{T} C^{*} (SE) p_{T}$.  Therefore, $p ( C^{*} (E) \otimes \mathbb{K} ) p \cong C^{*} (F)$.
\end{proof}

The following theorem  answers a question asked by George A.~Elliott at the NordForsk Closing Conference at the Faroe Islands, May 2012.

\begin{theorem}\label{t:morita-ck-algebras}
Let $A$ be a unital $C^*$-algebra.
\begin{itemize}
\item[(1)] If $A$ is stably isomorphic to a Cuntz-Krieger algebra, then $A$ is isomorphic to a Cuntz-Krieger algebra.

\item[(2)] Let $A$ be a unital, nuclear, separable $C^{*}$-algebra with finitely many ideals and let $X = \mathrm{Prim} ( A )$.  If $A \otimes \mathcal{O}_{\infty}$ is $KK_{X}$-equivalent to a Cuntz-Krieger algebra with real rank zero and primitive ideal space $X$, then $A \otimes \mathcal{O}_{\infty}$ is isomorphic to a Cuntz-Krieger algebra of real rank zero.
\end{itemize}
\end{theorem}

\begin{proof}
We first prove (1).  Let $B$ be a Cuntz-Krieger algebra such that $A \otimes \mathbb{K} \cong B \otimes \mathbb{K}$.  Note that $B = C^{*} ( F )$ such that $F$ is a finite graph with no sinks and no sources.  By Theorem~5.2 of \cite{as:geometric-class}, collapsing a regular vertex that is not a base point of a cycle of length one preserves stable isomorphism classes.  Therefore, since $F$ is a finite graph with no sinks and no sources, we can apply Theorem~5.2 of \cite{as:geometric-class} a finite number of times to get a finite graph $E$ with no sinks and no sources, and every vertex of $E$ is a base point of at least one cycle of length one, such that $C^{*} (F) \otimes \mathbb{K} \cong C^{*} ( E ) \otimes \mathbb{K}$.  Hence, $A \otimes \mathbb{K} \cong C^{*} (E) \otimes \mathbb{K}$.  Let $\phi\colon A \otimes \mathbb{K} \to C^{*} (E) \otimes \mathbb{K}$ be an isomorphism.

Let $\{ e_{ij} \}_{ i ,j }$ be a system of matrix units for $\mathbb{K}$.  Since $1_{A} \otimes e_{11}$ is a norm-full projection in $A \otimes \mathbb{K}$, $p = \phi ( 1_{A} \otimes e_{11} )$ is a norm-full projection in $C^{*} (E) \otimes \mathbb{K}$.  By Proposition~\ref{p:full-corner-CK-algebras}, $p ( C^{*} (E) \otimes \mathbb{K} ) p$ is isomorphic to a Cuntz-Krieger algebra.  Note $( 1_{A} \otimes e_{11} )( A \otimes \mathbb{K} ) ( 1_{A} \otimes e_{11} ) \cong p ( C^{*} (E) \otimes \mathbb{K} ) p$ and $A \cong  ( 1_{A} \otimes e_{11} )( A \otimes \mathbb{K} ) ( 1_{A} \otimes e_{11} )$.  Therefore, $A$ is isomorphic to a Cuntz-Krieger algebra.

\medskip

We will now use (1) to prove (2).  Let $B$ be a Cuntz-Krieger algebra with real rank zero such that $A \otimes \mathcal{O}_{\infty}$ is $KK_{X}$-equivalent to $B$ and $\mathrm{Prim}(B)\cong X$.  By  Folgerung 4.3 of \cite{kirchberg}, $A \otimes \mathcal{O}_{\infty} \otimes \mathbb{K} \cong B \otimes \mathbb{K}$.  Therefore, $A \otimes \mathcal{O}_{\infty}$ is a unital $C^{*}$-algebra stably isomorphic to a Cuntz-Krieger algebra with real rank zero.  By (1), we have that $A \otimes \mathcal{O}_{\infty}$ is isomorphic to a Cuntz-Krieger algebra.  Since $A \otimes \mathcal{O}_{\infty}$ is stably isomorphic to a $C^{*}$-algebra with real rank zero, $A \otimes \mathcal{O}_{\infty}$ has real rank zero.  Therefore, $A \otimes \mathcal{O}_{\infty}$ is isomorphic to a Cuntz-Krieger algebra with real rank zero.
\end{proof}

\begin{corollary} \label{cor:matrices}
Let $A$ be a $C^{*}$-algebra.  Then the following are equivalent.
\begin{itemize}
\item[(1)] $A$ is a Cuntz-Krieger algebra.

\item[(2)] $\mathsf M_{n}(A)$ is a Cuntz-Krieger algebra for all $n\in\mathbb \N$.

\item[(3)] $\mathsf{M}_{n} (A)$ is a Cuntz-Krieger algebra for some $n \in \N$.
\end{itemize}
\end{corollary}

\begin{proof}
(1) implies (2) follows from Theorem~\ref{t:morita-ck-algebras}.  (2) implies (3) is obvious.  Suppose $\mathsf{M}_{n} (A)$ is a Cuntz-Krieger algebra for some $n \in \N$.  In particular, $\mathsf{M}_{n} (A)$ is a unital $C^{*}$-algebra with $1_{\mathsf{M}_{n} (A)} = [ x_{ij} ]$.  A computation shows that $x_{11}$ is a multiplicative identity for $A$.  Therefore, $A$ is a unital $C^{*}$-algebra.  Since $A \otimes \mathbb{K} \cong \mathsf{M}_{n} ( A ) \otimes \mathbb{K}$ and since $\mathsf{M}_{n} (A)$ is a Cuntz-Krieger algebra, by Theorem~\ref{t:morita-ck-algebras}, $A$ is a Cuntz-Krieger algebra.
\end{proof}

\begin{corollary} \label{cor}
Let $A$ be a Cuntz-Krieger algebra.  
\begin{itemize}
\item[(1)] If $p$ is a nonzero projection in $A$, then $p A p$ is isomorphic to a Cuntz-Krieger algebra.

\item[(2)] If $p$ is a nonzero projection in $A \otimes \mathbb{K}$, then $p( A \otimes \mathbb{K} ) p$ is isomorphic to a Cuntz-Krieger algebra.
\end{itemize}
\end{corollary}

\begin{proof}
We first prove (1) in the case when $p$ is a norm-full projection.  Suppose $p$ is a norm-full projection.  By Corollary~2.6 of \cite{lb-her-algs}, $pAp \otimes \mathbb{K} \cong A \otimes \mathbb{K}$.  Therefore, $pAp$ is a unital $C^{*}$-algebra that is stably isomorphic to a Cuntz-Krieger algebra.  By Theorem~\ref{t:morita-ck-algebras}, $p A p$ is isomorphic to a Cuntz-Krieger algebra.

We now prove the general case in (1).  Let $A = C^{*} (E)$ where $E$ is a finite graph with no sinks and no sources.  Let $p$ be a nonzero projection of $A$.  Set 
\begin{align*}
I = \text{ the ideal in $C^{*} (E)$ generated by $p$}.
\end{align*}  
Note that $p A p \subseteq I$ which implies that $p A p \subseteq p I p$.  Since $p I p \subseteq p A p$, we have that $p A p = p I p$.  Thus, $p I p$ is a norm-full hereditary subalgebra of $I$.  By Corollary~2.6 of \cite{lb-her-algs}, $p I p \otimes \mathbb{K} \cong I \otimes \mathbb{K}$.

Since $I$ is generated by a projection $p$, by Theorem~7.3 and the proof of Theorem~5.3 of \cite{amp:nonstablekthy}, $I$ is a gauge-invariant ideal of $C^{*} (E)$.  Thus, by Theorem~3.7 of \cite{bhrs:iccig}, there exists a hereditary saturated subset $H$ of $E^{0}$ such that $I_H$ is $I$, see Definition~\ref{d:hereditary}.

 Let $E_{H} = ( H , s_{E}^{-1} ( H ), r_{E} , s_{E} )$.  By Proposition~3.4 of \cite{bhrs:iccig}, $I_{H} \otimes \mathbb{K} \cong C^{*} ( E_{H} ) \otimes \mathbb{K}$.  Note that $E_{H}$ is a finite graph with no sinks.  By Proposition~3.1 of \cite{as:geometric-class}, we may continue to remove the sources to obtain a finite graph $F$ with no sinks and no sources such that $C^{*} ( E_{H} ) \otimes \mathbb{K} \cong C^{*} (F ) \otimes \mathbb{K}$.  Hence, $C^{*} (F)$ is a Cuntz-Krieger algebra and $p A p = p I p$ is a unital $C^{*}$-algebra that is stably isomorphic to $C^{*} (F)$.  By Theorem~\ref{t:morita-ck-algebras}, $p A p$ is isomorphic to a Cuntz-Krieger algebra.  
 
We now prove (2).  Let $p$ be a nonzero projection in $A \otimes \mathbb{K}$.  Recall that $A = C^{*} (E)$, where $E$ is a finite graph with no sinks and no sources.  By Theorem~\ref{amp:nonstablekthy}, there exists a non-empty subset $S$ of $E^{0}$ and a collection of positive integers $\{ m_{v} \}_{v \in S}$ such that $p$ is Murray-von Neumann equivalent to $\sum_{ v \in S } m_{v} p_{v}$.  Set $q = \sum_{ v \in S } p_{v}$.  Then $q$ is a nonzero projection in $A$ and by (1), we have that $q A q \cong C^{*} (F)$ for some finite graph $F$ with no sinks and no sources.  By Theorem~5.3 of \cite{amp:nonstablekthy}, $p$ and $q \otimes e_{11}$ generate the same ideal of $A \otimes \mathbb{K}$.  Hence, $q A q \otimes \mathbb{K} \cong ( q \otimes e_{11} ) A \otimes \mathbb{K} ( q \otimes e_{11} ) \cong p ( A \otimes \mathbb{K} ) p \otimes \mathbb{K}$.  Therefore, $p ( A \otimes \mathbb{K} ) p$ is stably isomorphic to a Cuntz-Krieger algebra.  By Theorem~\ref{t:morita-ck-algebras}, $p ( A \otimes \mathbb{K} ) p$ is isomorphic to a Cuntz-Krieger algebra.
\end{proof}

\begin{corollary} \label{c:sem}
Let $A$ be a Cuntz-Krieger algebra.  If $p$ is a projection in $A \otimes \mathbb{K}$, then $p( A \otimes \mathbb{K} ) p$ is semiprojective.  If $p$ is a projection in $A$, then $p A p$ is semiprojective.
\end{corollary}
\begin{proof}
This follows from Corollary~\ref{cor} since by Corollary~2.24 of~\cite{blackadar} all Cuntz-Krieger algebras are semiprojective and by Corollary~2.29 of~\cite{blackadar} all stabilized Cuntz-Krieger algebras are semiprojective.
\end{proof}

\section{Acknowledgements}
The authors are grateful to George A.~Elliott for asking such inspiring questions.
The authors also wish to thank S{\o}ren Eilers, Adam S{\o}rensen, and Mark Tomforde for helpful conversations that have led to the improvement of our results.  The second named author is grateful to S{\o}ren Eilers and the Department of Mathematical Sciences at the University of Copenhagen for providing the dynamic research environment where this work was initiated during the Spring of 2012.

This research was supported by the Danish National Research Foundation through the Centre for Symmetry and Deformation (DNRF92) at University of Copenhagen, and by the NordForsk research network ``Operator Algebras and Dynamics'' (grant \#11580).

\def\cprime{$'$}

\end{document}